\documentclass[11pt, reqno]{amsart}

\usepackage{latexsym,fullpage}
\usepackage{amssymb}
\usepackage{amsfonts}
\usepackage{amsmath}
\usepackage{amsthm}
\usepackage[english]{babel}

\newtheorem{thm}{Theorem}[section]
\newtheorem{lem}[thm]{Lemma}
\newtheorem{cor}[thm]{Corollary}
\newtheorem{pro}[thm]{Proposition}

\theoremstyle{definition}

\newtheorem{exmp}[thm]{Example}

\newtheorem{rem}[thm]{Remark}

\numberwithin{equation}{section}

\def\ni{\noindent}

\def\H{{\text {\bf H}}}

\def\ra{\longrightarrow}

\def\pp{{\mathbb P}^2}

\def\ra{\rightarrow}

\def\ds{\displaystyle}
\def\Gin{{\rm Gin}}
\def\soc{{\rm soc}}
\def\FF{{\mathcal F}}
\def\GG{{\mathcal G}}
\def\pp{{\phantom{-}}}
\def\be{\bfseries\em}

\begin{document}


\title{Artinian level algebras of codimension $3$}

\author[J. Ahn] {Jeaman Ahn${}^1$}
\address{Department of Mathematics Education, Kongju National University, 182, Shinkwan-dong, Kongju, Chungnam 314-701, Republic of Korea}
\email{jeamanahn@kongju.ac.kr}
\thanks{${}^1$This research was supported by Basic Science Research Program through the National Research Foundation of Korea (NRF) funded by the Ministry of Education, Science, and Technology (No. 2010-0025762).}

\author[Y.S. SHIN]{Yong Su Shin${}^{2,*}$}
\address{Department of Mathematics, Sungshin Women's University, Seoul, Korea, 136-742}
\email{ysshin@sungshin.ac.kr }
\thanks{${}^2$This research was supported by a grant from Sungshin Women's University in 2009 (2009-2-21-003/1).}
\thanks{${}^*$Corresponding author}

\begin{abstract}
In this paper, we continue the study of which $h$-vectors $\H=(1,3,\dots, h_{d-1}, h_d, h_{d+1})$  can be the Hilbert function of a level algebra by investigating Artinian level algebras of codimension 3 with the condition  $\beta_{2,d+2}(I^{\rm lex})=\beta_{1,d+1}(I^{\rm lex})$, where $I^{\rm lex}$ is the lex-segment ideal associated with an ideal $I$. Our approach is to adopt an homological method called {\it Cancellation Principle}: the minimal free resolution of $I$ is obtained from that of $I^{\rm lex}$ by canceling some adjacent terms of the same shift.

 We prove that when $\beta_{1,d+2}(I^{\rm lex})=\beta_{2,d+2}(I^{\rm lex})$, $R/I$ can be an Artinian level $k$-algebra only if either $h_{d-1}<h_d<h_{d+1}$ or  $h_{d-1}=h_d=h_{d+1}=d+1$ holds. We also
show that for $\H=(1,3,\dots, h_{d-1}, h_d, h_{d+1})$, the Hilbert function of an Artinian algebra of codimension $3$ with the condition  $h_{d-1}=h_d<h_{d+1}$, \begin{itemize}
 \item[(a)] if $h_d\leq 3d+2$, then $h$-vector $\H$ cannot be level, and 
 \item[(b)] if $h_d\geq 3d+3$, then there is a level algebra with Hilbert function $\H$ for some value of $h_{d+1}$.
\end{itemize} 

\end{abstract}


\maketitle

\section{Introduction}

Let $R=k[x_1,\dots,x_n]$ be an $n$-variable polynomial ring over a field $k$ of characteristic zero, and $I$ be a homogeneous ideal of $R$. The numerical function
$$
\H_{R/I}(t):= \dim _k R_t - \dim _k I_t
$$
is called the {\em Hilbert function} of  the ring $R/I$. 

Recall that if $n$ and $i$ are positive integers, then $n$ can be written uniquely in the form  $$
n_{(i)} =\binom{n_i}{i} + \binom{n_{i-1}}{i-1}+\cdots+\binom{n_j}{j} ,
$$
where $n_i>n_{i-1}>\cdots >n_j\ge j\ge 1$ (see Lemma 4.2.6, \cite{BH}).

Following \cite{BG}, we define, for any integers $a$ and $b$,
$$
\big(n_{(i)}\big)^a_b =\binom{n_i+a}{i+b} + \binom{n_{i-1}+a}{i-1+b}+\cdots+\binom{n_j+a}{j+b} 
$$
where $\binom{m}{n}=0$ for either $m<n$ or $n<0$. 

Let $\H=(h_0, h_1,\dots,h_i,\dots)$ be a sequence of non-negative integers. We say that $\H$ is an $O$-sequence if $h_0=1$ and $h_{i+1}\le ((h_i)_{(i)})^1_1$ for all $i\ge 1$. Given an $O$-sequence $\H=(h_0,h_1,\dots)$, we define the {\em first difference} of $\H$ as 
$$
\Delta \H=(h_0,h_1-h_0,h_2-h_1,h_3-h_2,\dots).
$$

If $A=R/I$ is an {\em Artinian $k$-algebra}, then we associate the graded algebra $A=k\oplus A_1\oplus \cdots \oplus A_s$, $(A_s\ne 0)$ with a vector of nonnegative integers, which is an $(s+1)$-tuple, called the {\em$h$-vector} of $A$ and denoted by 
$$
\H_A:=\H=(h_0,h_1,\dots,h_s),
$$
where $h_i=\dim_k A_i$. We call $s$ the {\em socle degree} of $A$. The {\em socle} of $A$
is defined to be the annihilator of the maximal homogeneous ideal, namely
$$
{\rm Ann}_A(m):=\{a\in A \mid ma=0\} \text{ where } m=\sum_{i=1}^s A_i. 
$$


Let $\FF$ be the graded minimal free resolution of an homogeneous ideal $I\subset R$, i.e.,
$$
\begin{array}{ccccccccccccccccccccccccccccccccccccccccccc}
\FF: & 0 & \ra & \FF_{n-1} & \ra & \FF_{n-2} & \ra & \cdots
     & \ra & \FF_1 & \ra & \FF_0& \ra & I & \ra & 0,
\end{array}
$$
where
$
\FF_i
 =  \bigoplus^{\gamma_{i}}_{j=1} R^{\beta_{i,j}}(-j).
$
The numbers $j$ are called the {\em shifts} associated to $I$, and the numbers $\beta_{i,j}$ are called the {\em graded Betti numbers} of $I$. When we need to emphasize the ideal $I$, we shall use $\beta_{i,j}(I)$ for $\beta_{i,j}$.

An algebra $A$ is called  an {\em Artinian  level algebra} if the last module $\mathcal F_{n-1}$ in the minimal free resolution of $A$ is of the form $R(-s)^a$, where $s$ and $a$ are positive integers. We also say that a numerical sequence $\H = 
(h_0 , h_1 , \dots, h_{s-1},  h_s)$ is a {\em level $O$-sequence} if there is an Artinian level algebra $A$ with the Hilbert function $\H$. 

As for the level $O$-sequence, an interesting question is how to determine if a given numerical sequence is a level $O$-sequence. A great deal of research has been conducted with the aim of answering to this question (see e.g., \cite{AM, AS,  BI, BG, BL, BZ:1, BZ:2, CI, GHMS, GHS:4, GL, I-K, M, Za1, Za2}). In particular, there is an excellent broad overview of level algebras in the memoir \cite{GHMS}. Despite this, it is sometimes distressingly difficult to find ones with specific desired properties, and several interesting problems are still open.


In \cite{AS}, we proved that an Artinian algebra with Hilbert function $\H = 
(1 ,3, h_2 , \dots, h_{d-1}, h_d, h_{d+1} )$ with the condition  $h_{d-1} > h_d = h_{d+1}$ cannot be level if $h_d \leq 2d + 3$, and proved that if $h_d \geq 2d + 4$ then there is a level O- 
sequence of codimension 3 with Hilbert function $\H$ for some value of $h_{d-1}$. 
To prove the result, we used the cancellation principle saying that the minimal free resolution of $I$ is obtained from that of either $\Gin(I)$ or $I^{\rm lex}$ by canceling some adjacent terms of the same shift, where $\Gin(I)$ is the generic initial ideal of $I$ with respect to the reverse lexicographic order and $I^{\rm lex}$ is the lex-segment ideal associated with an ideal $I$ (see \cite{G}, \cite{P}). 

By the cancellation principle, one knows that $\H=(1, 3,\dots,h_s)$ cannot be a level $O$-sequence if  $\beta_{1,d+2}(\Gin(I))<\beta_{2,d+2}(\Gin(I))$ or $\beta_{1,d+2}(I^{\rm lex})<\beta_{2,d+2}(I^{\rm lex})$ for some $d<s$. However, the problem that we wish to solve is to determine whether a given $h$-vector can be a level $O$-sequence with the condition  $\beta_{1,d+2}(\Gin(I))=\beta_{2,d+2}(\Gin(I))$ or $\beta_{1,d+2}(I^{\rm lex})=\beta_{2,d+2}(I^{\rm lex})$. In this case, it is known that an Artinian algebra $A=R/I$ of codimension 3 with Hilbert function $\H=(1,3,\dots,h_s)$ cannot be a level algebra (Theorem 3.14, \cite{AS}) if 
\begin{itemize}
   \item[(a)] $\beta_{1,d+2}(\Gin(I))=\beta_{2,d+2}(\Gin(I))$  for some $d<s$, or
   \item[(b)] $\beta_{1,d+2}(I^{\rm lex})=\beta_{2,d+2}(I^{\rm lex})$ with the condition  $h_{d-1}>h_d=h_{d+1}$ for some $d<s$.
\end{itemize}

From this result, we wish to determine what Hilbert functions can be an Artinian level $O$-sequences with the condition 
\begin{equation}\label{eq:20110324}
\beta_{1,d+2}(I^{\rm lex})=\beta_{2,d+2}(I^{\rm lex}) \text{ for some } d<s.
\end{equation} 

We first prove that $R/I$ can be an Artinian level $k$-algebra only if either $h_{d-1}<h_d<h_{d+1}$ with $\Delta h_{d}=\Delta h_{d+1}$, or  $h_{d-1}=h_d=h_{d+1}=d+1$ 
with the condition (\ref{eq:20110324}) (see Theorem~\ref{thm:4} and  Corollary~\ref{T:20101110-404}). Using these results, we also
prove that for $\H=(1,3,\dots, h_{d-1}, h_d, h_{d+1})$, the Hilbert function of an Artinian algebra of codimension $3$ with the condition  $h_{d-1}=h_d<h_{d+1}$, 
\begin{itemize}
 \item[(a)] if $h_d\leq 3d+2$, then $h$-vector $\H$ cannot be level, and 
 \item[(b)] if $h_d\geq 3d+3$, then there is a level algebra with Hilbert function $\H$ for some value of $h_{d+1}$.
\end{itemize}

In Section 2, we introduce some preliminary results and background materials which will be used throughout the remaining part of the paper. 
In Section 3, we make use of {\it cancellation in resolutions} to study Artinian level algebras of codimension 3 with the condition (\ref{eq:20110324}). Finally, Section 4 is devoted to investigate Artinian level or non-level algebras with the condition $h_{d-1}=h_d<h_{d+1}$. 

We use a computer program CoCoA \cite{RABCP} to build some of examples (e.g., Examples~\ref{EX:20101110-401} and \ref{EX:20101110-408}), with  the fact that a differentiable $O$-sequence can always be a truncation of an Artinian Gorenstein $O$-sequence (see \cite{GHMS, GHS:2, GHS:4, GKS, GPS, GS:1, Ha:1, Ha:2}).

\section{Background and Preliminary Results}

In this section, we introduce some important results and recall some results of Macaulay, Green, and Stanley.



\begin{thm}[\cite{G1}, Chapter 5 in \cite{KR}]\label{T:201}
Let $L$ be a general linear form in $R$ and we denote by $h_d$ the degree $d$ entry of the Hilbert function of $R/I$ and $\ell_d$ the degree $d$ entry of the Hilbert function of $R/(I,L)$. Then, we have the following inequalities.
\begin{enumerate}
\item [(a)] {\rm Macaulay's Theorem:} $h_{d+1} \leq \left((h_d)_{(d)}\right)^{1}_{1}$.
\item [(b)] {\rm Green's Hyperplane Restriction Theorem:} $\ell_d \leq \left((h_{d})_{(d)}\right)^{-1}_{\phantom{-}0}.$
\end{enumerate}
\end{thm}

For any homogeneous ideal $I$ of $R=k[x_1,\dots,x_n]$, note that the Hilbert function does not change by passing to $\Gin(I)$ or $I^{\rm lex}$, and we have
$$
\beta_{q,i}(I) \le \beta_{q,i}(\Gin(I)) \le \beta_{q,i}(I^{\rm lex})
$$
\text{ (see \cite{AM, Bi, G, Hul, Pa})}. In particular, if $\beta_{q,i}(\Gin(I))=0$ or $\beta_{q,i}(I^{\rm lex})=0$, then $\beta_{q,i}(I)=0$.

In \cite{HT}, they introduced the {\em $s$-reduction number} $r_s(R/I)$ of $R/I$ and have shown the following lemma.

\begin{lem}[\cite{AM, HT}]\label{Reduction Number_2}
For a homogeneous ideal $I$ of $R$ and for $s\geq \dim(R/I)$, the $s$-reduction number $r_s(R/I)$ is given by 
\begin{align*}
r_s(R/I) = &~\min\{\ell \mid \textup{Hilbert function of }R/(I+J) \textup{ vanishes in degree }\ell+1\}\\
            = &~\min\{\ell \mid x_{n-s}^{\ell+1}\in \Gin(I)\}\\
            = &~ r_s(R/\Gin(I))
\end{align*}
where $J$ is generated by $s$ general linear forms of $R$.
\end{lem}

Now we continue to introduce some lemmas and theorems that will be used to prove the main results of this paper. 

\begin{lem}[Lemma 3.2, \cite{AS}] \label{L:006}
Let $A=R/I$ be an Artinian algebra and let $L$ be a general
linear form. Suppose that $\dim_k((I:L)/I)_d > (n-1)\dim_k((I:L)/I)_{d+1}
$ for some $d>0$. Then $A$ has a socle element in degree $d$.
\end{lem}

We denote by $\mathcal
G(I)$ the set of minimal (monomial) generators of $I$ and
$\mathcal G(I)_d$ the elements of $\mathcal G(I)$ having degree
$d$. For a monomial $T=x_1^{a_1}\cdots x_n^{a_n}\in R$, define
$$
m(T):=\max\{ j \mid a_j>0\}. 
$$

\begin{thm}[Eliahou-Kervaire, \cite{EK}]\label{T:012}
Let $I$ be a stable monomial ideal of $R$. Then we have
$$
\beta_{q,i}(I)=\sum_{T\in \mathcal
G(I)_{i-q}}\binom{m(T)-1}{q}.
$$
\end{thm}

\begin{lem}[Lemma 3.8, \cite{AS}] \label{L:013}
Let $J$ be a stable ideal of $R$. Then we have
\begin{align*}
\dim_k \left((J:x_n)/J\right)_{d-1} =  \big|\{T \in \mathcal G(J)_{d}
\mid x_n\text{ divides } T \}\big|.
\end{align*}
\end{lem}




%
%

We now recall the well known result in \cite{G}, from which the generic initial ideal with respect to the degree reverse lexicographic order is extremely well-suited to the quotient by general linear forms.

\begin{pro} [Corollary 2.15, \cite{G}] \label{P:20101107-211}
Consider the degree reverse lexicographic order on the monomials of $R=k[x_1,\dots,x_n]$. Let $I$ be a homogeneous ideal in $R$ and $H$ be a general linear form
in $R$. Then
$$
\Gin(I+(H)/(H))=(\Gin(I)+(x_n))/(x_n).
$$
\end{pro}

\begin{rem} \label{R:20101107-212} Let $I$ be  a homogeneous ideal  of $R=k[x_1,\dots,x_n]$ and $L$ be a general linear form in $R$. Using Proposition~\ref{P:20101107-211} and the exact sequence
$$
\begin{matrix}
0 & \rightarrow & R/(I:L)(-1) & \overset{\times L}{\rightarrow} & R/I & \rightarrow & R/(I,L) & \ra &  0,
\end{matrix}
$$
we have  
$$
\begin{array}{rllllllllllllllllllll}
\dim_k(I:L)_t 
& = & \dim_k{(\Gin(I):x_n)}_t \\
\H(R/(I,L),t)
& = & \H(R/(\Gin(I),x_n),t) \\
\dim_k((I:L)/I)_t 
& = & \dim_k{((\Gin(I):x_n)/\Gin(I))}_t
\end{array}
$$
for $t\ge 0$.
\end{rem}


\begin{rem}\label{R:20101107-213} 
Let $I^{\rm lex}$ be the lex-segment ideal associated with a homogeneous ideal $I$ in $R=k[x_1,\ldots, x_n]$ and $L$ be a general linear form in $R$. Then, by Theorem~2.4,  \cite{ERV}, we have the following equality
$$
\H(R/(I^{\rm lex},L),d)=(\H(R/I^{\rm lex},d)_{(d)})^{-1}_{\phantom{-}0}.
$$
In this case, we may assume that $x_n$ is general with respect to $I^{\rm lex}$. Indeed, for $d\ge 1$, we have 
$$
\begin{array}{lllllllllllllllllllllll}
(\H(R/I,d)_{(d)})^{-1}_{\phantom{-}0}
& = & (\H(R/I^{\rm lex},d)_{(d)})^{-1}_{\phantom{-}0} \\
& = & \H(R/(I^{\rm lex},L),d) & (\text{by Theorem~2.4, \cite{ERV}})\\
& = & \H(R/(\Gin(I^{\rm lex}),x_n),d) & (\text{by Proposition~\ref{P:20101107-211} and Remark~\ref{R:20101107-212}})\\
& = & \H(R/(I^{\rm lex},x_n),d) & (\text{by Lemma~2.3, \cite{C}}).
\end{array}
$$
\end{rem}


The following lemma shows that we can write some of Betti numbers of the lex-segment ideal associated with a height three ideal $I$ with respect to binomial expansion of the Hilbert function.

\begin{lem} \label{L:20101104-208}
Let $A=R/I$ be an Artinian  ring of codimension $3$ with Hilbert function $\H=(h_0,h_1,\dots,h_s)$. Suppose that $h_d < \binom{2+d}{2}$. Then, we have 
\begin{itemize}
 \item[(a)] $\beta_{2,d+2}(I^{\rm lex})=h_{d-1}-h_{d}+((h_{d})_{(d)})^{-1}_{\pp 0}$.
 \item[(b)]  $\beta_{1,d+2}(I^{\rm lex})={((h_d)}_{(d)})^{1}_{1}+h_d-2h_{d+1}+((h_{d+1})_{(d+1)})^{-1}_{\pp 0}$.
\end{itemize}
\end{lem}
\begin{proof}
(a) From the following exact sequence
$$
0 \rightarrow  ( (I^{\rm lex}:x_3) /I^{\rm lex} )_{d-1} \rightarrow
(R/I^{\rm lex})_{d-1} \stackrel{\times x_3}{\longrightarrow} (R/I^{\rm lex})_{d}
\rightarrow (R/(I^{\rm lex}, x_3))_{d} \rightarrow 0,
$$
we have 
\begin{equation}\label{EQ:20101103-202}
\begin{array}{llllllllllll}
 \beta_{2, d+2}(I^{\rm lex}) 
 & = & {\ds \sum}_{T\in \mathcal
G(I^{\rm lex})_{d}}\binom{m(T)-1}{2} & (\text{by Theorem~\ref{T:012}})\\[1.5ex]
& = & \dim( (I^{\rm lex}:x_3) /I^{\rm lex} )_{d-1} & (\text{by Lemma~\ref{L:013}} )\\
& = & h_{d-1}-h_{d} + (h_{d})^{-1}_{\pp 0} & (\text{by Remark~\ref{R:20101107-213}})
\end{array}
\end{equation}
as we wished.

\medskip

\ni
(b) Since $I^{\rm lex}$ is a lex-segment ideal associated with an ideal $I$ of $R$, we see that $$\beta_{0, d+1}(I^{\rm lex})={((h_d)}_{(d)})^{1}_{1}-h_{d+1}.$$ Let $\mathcal G(I^{\rm lex})_{d+1}$ be the set of minimal generators of $I^{\rm lex}$ in degree $d+1$. Then, 
$$
\begin{array}{lllllllllllllll}
\beta_{1,d+2}(I^{\rm lex})
& = & {\ds \sum}_{T\in \mathcal G(I^{\rm lex})_{d+1},\,m(T)=2}{ {\binom{1}{1}}}+{\ds \sum}_{T\in \mathcal G(I^{\rm lex})_{d+1},\,m(T)=3}{ {\binom{2}{1}}} & (\text{by Theorem~\ref{T:012}})\\[2ex] 
& = & {\ds \sum}_{T\in \mathcal G(I^{\rm lex})_{d+1},\,m(T)=2}{ {\binom{1}{1}}}+2 \Big[{\ds \sum}_{T\in \mathcal G(I^{\rm lex})_{d+1},\,m(T)=3}{ {\binom{1}{1}}}\Big] \\[2ex] 
& = & \Big[{\ds \sum}_{T\in \mathcal G(I^{\rm lex})_{d+1},\,m(T)=2}{ {\binom{1}{1}}}+{\ds \sum}_{T\in \mathcal G(I^{\rm lex})_{d+1},\,m(T)=3}{ {\binom{1}{1}}}\Big] +\\[2ex] 
&  & {\ds \sum}_{T\in \mathcal G(I^{\rm lex})_{d+1},\,m(T)=3}{ {\binom{1}{1}}} \\[2ex]
& = & \big|\mathcal G(I^{\rm lex})_{d+1}\big| + \big|\{T\in \mathcal G(J)_{d+1}
\mid \,x_3\text{ divides } T\, \}\big|\\[1.5ex]
&   & \big(\text{since } h_d < \binom{2+d}{2}, \ x_1^{d+1}\notin \mathcal G(I^{\rm lex})_{d+1} \big)\\[1ex]

& = & \big|\mathcal G(I^{\rm lex})_{d+1}\big| +\dim_k ((I^{\rm lex}:x_3)/I^{\rm lex})_d & 
(\text{by Lemma~\ref{L:013}})\\[1ex]
& = & \big({((h_d)}_{(d)})^{1}_{1}-h_{d+1}\big)+\beta_{2,d+3}(I^{\rm lex}) & (\text{by equation~\eqref{EQ:20101103-202}})\\[.5ex]
& = & {((h_d)}_{(d)})^{1}_{1}-h_{d+1}+h_d-h_{d+1}+((h_{d+1})_{(d+1)})^{-1}_{\pp0} & (\text{by Lemma~\ref{L:20101104-208} (a)}) \\
& = & {((h_d)}_{(d)})^{1}_{1}-2h_{d+1}+h_d+((h_{d+1})_{(d+1)})^{-1}_{\pp0}  ,
\end{array}
$$
as we wanted to prove.
\end{proof}

\section{$O$-sequences with the Condition on $\beta_{1,d+2}(I^{\rm lex})=\beta_{2,d+2}(I^{\rm lex})$}  

First, we investigate  if some Artinian $O$-sequence with the condition 
$$
\beta_{1,d+2}(I^{\rm lex})=\beta_{2,d+2}(I^{\rm lex})
$$
is level. 

\begin{lem}\label{L:20101104-209}
Let  $A=R/I$ be an Artinian  ring of codimension $3$ with Hilbert function $\H=(h_0,h_1,\dots,h_s)$. Suppose that for some $d<s$,
\begin{itemize}
 \item[(a)] $\beta_{1,d+2}(I^{\rm lex})=\beta_{2,d+2}(I^{\rm lex})>0$, and 
 \item[(b)] $\beta_{2,d+3}(I^{\rm lex})>0$.
\end{itemize}  
Then $A$ is not level.
\end{lem}

\begin{proof}
Assume that there exists an Artinian level algebra $A$ with Hilbert function $\H$, and let $\bar I=(I_{\le d+1})$ and $\bar A=R/\bar I$.  Then we have   
$$
\begin{array}{lllllllllllllll}
\beta_{1,d+2}(\bar I^{\rm lex})=\beta_{1,d+2}(I^{\rm lex}), \\  
\beta_{2,d+2}(\bar I^{\rm lex})=\beta_{2,d+2}(I^{\rm lex}), & \text{and}\\  
\beta_{2,d+3}(\bar I^{\rm lex})=\beta_{2,d+3}(I^{\rm lex}).
\end{array}
$$
Hence, the assumption $\beta_{1,d+2}(I^{\rm lex})= \beta_{2,d+2}(I^{\rm lex})$ and $\beta_{2,d+3}(I^{\rm lex})>0$ implies that
\begin{align}\label{eq:2001}
\beta_{1,d+2}(\bar I^{\rm lex})&=\beta_{2,d+2}(\bar I^{\rm lex}), \quad \text{and} \\
\beta_{2,\,d+3}(\bar I^{\rm lex})&=\beta_{2,\,d+3}(I^{\rm lex})>0  .
\end{align}
Since $A$ is level and $I_t=(\bar I)_t$ for every $t\le d+1$, 
\begin{equation} \label{EQ:006} 
0 =   \beta_{2,d+2}(I) = \dim_k \soc(A)_{d-1} = \dim_k \soc(\bar A)_{d-1} = \beta_{2,d+2}(\bar I).
\end{equation}

Furthermore, using Lemma 2.9 in \cite{AS},  we have the following equality
$$
\beta_{1,d+2}(\bar I^{\rm lex})- \beta_{1,d+2}(\bar I)
  = [\beta_{0,d+2}(\bar I^{\rm lex})- \beta_{0,d+2}(\bar I)]+
[\beta_{2,d+2}(\bar I^{\rm lex})- \beta_{2,d+2}(\bar I)].
$$
Hence it follows from equations~\eqref{eq:2001} and \eqref{EQ:006} that  
\[- \beta_{1,d+2}(\bar I)
=\beta_{0,d+2}(\bar I^{\rm lex})- \beta_{0,d+2}(\bar I)\geq 0,
\]
which means that $\beta_{0,d+2}(\bar I^{\rm lex})= \beta_{0,d+2}(\bar I)=0$ since $\bar I$ is generated in degree $d+1$. This concludes from Theorem~\ref{T:012} that 
$$\beta_{0,d+2}(\bar I^{\rm lex})=
\beta_{1,d+3}(\bar I^{\rm lex})=0.$$ 

In other words, any cancellation on shifts is  impossible in the last free module of the minimal free resolution of $R/\bar I^{\rm lex}$ in degree $d$, and thus we have that $\beta_{2,\,d+3}(\bar I^{\rm lex})=\beta_{2,\,d+3}(\bar I)>0$. Hence $\bar A$ has a socle element in degree $d$, and so does $A$, which is a contradiction, as we wanted. 
\end{proof}

\begin{exmp} Consider an Artinian $O$-sequence $\H=(1, 3, 6, 10, 15, 16, 18)$. Then the minimal free resolution of $R/I^{\rm lex}$ with Hilbert function is 
$$
\begin{array}{lllllllllllllllllll}
0 & \ra &  R^2(-{\mathbf 7})\oplus R(-{\mathbf 8})\oplus R^{18}(-9) & \ra &  R^6(-6)\oplus R^2(-{\mathbf 7})\oplus R^{39}(-8) \\
& \ra &  R^5(-5)\oplus R(-6)\oplus R^{21}(-7) & \ra &  R \ \ra \ R/I^{\rm lex} \ \ra \ 0. 
\end{array}
$$
Then 
$$
\beta_{2,7}(I^{\rm lex}) = \beta_{1,7}(I^{\rm lex}) = 2 \quad \text{and} \quad
\beta_{2,8}(I^{\rm lex})=1.
$$
By Lemma~\ref{L:20101104-209},  any Artinian ring with Hilbert function $\H$ cannot be a level algebra. 
\end{exmp}


\begin{thm} \label{thm:4} Let $A=R/I$ be an Artinian ring of codimension 3 with Hilbert function $\H=(h_0,h_1,\dots,h_{d-1},h_d,h_{d+1})$. Suppose that
$$\beta_{1,d+2}(I^{\rm lex})=\beta_{2,d+2}(I^{\rm lex})>0 \text{ for some } d<s.$$
If $A$ is level, then 
\begin{itemize}
\item[(a)] $h_{d-1}= h_d = h_{d+1}=d+1$, or 
\item[(b)] $h_{d-1}<h_d<h_{d+1}$.
\end{itemize}
\end{thm}

\begin{proof} We shall prove this theorem using the contrapositive.  

\medskip

(a) Assume $h_{d-1}= h_d = h_{d+1}$.  First if $h_d\leq d$, then $((h_{d})_{(d)})^{-1}_{\pp0}=0$ and thus, by Lemma~\ref{L:20101104-208}, we have that
$$
\begin{array}{lllllllllllllll}
0<\beta_{2, d+2}(I^{\rm lex})
= h_{d-1}-h_d + ((h_{d})_{(d)})^{-1}_{\pp0}=0,
\end{array}
$$
which is impossible.

Second, if $h_d\geq d+2$, then $((h_{d+1})_{(d+1)})^{-1}_{\pp0}\geq 1$ and thus, by Lemma~\ref{L:20101104-208} again,
$$
\begin{array}{lllllllllllllll}
\beta_{2, d+3}(I^{\rm lex})
& = & h_{d}-h_{d+1} + ((h_{d+1})_{(d+1)})^{-1}_{\pp0}\\
& = & ((h_{d+1})_{(d+1)})^{-1}_{\pp0}>0.
\end{array}
$$
Hence, by Lemma~\ref{L:20101104-209}, $A$ is not level.

\medskip

(b) Now suppose $h_{d-1},h_d$ and $h_{d+1}$ are not the same, and (b) does not hold. There are five cases to be considered. 

\medskip
\ni
{\bfseries\em Case 1.} If $h_{d-1}>h_d=h_{d+1}$, then by Theorem~4.5 in \cite{AS}, $A$ is not level.

\medskip
\ni
{\bfseries\em Case 2.} If $h_{d-1}\ge h_d > h_{d+1}$, then $h_{d+1}<\binom{2+(d+1)}{2}$ and thus, by Lemma~\ref{L:20101104-208},
$$
\begin{array}{llllllllllllllll}
\beta_{2,d+3}(I^{\rm lex})
& = & h_d-h_{d+1}+((h_{d+1})^{-1}_{\pp0}\\
& \ge & h_d-h_{d+1} \\
& > & 0.
\end{array}
$$
Hence, by Lemma~\ref{L:20101104-209}, $A$ is not level.



\medskip
\ni
{\be Case 3.} Suppose that $h_{d-1}\geq h_{d}< h_{d+1}$. For this case, we shall use the reduction number $r_1(A)$. 

Assume $r_1(A)<d$. Note that, for a general linear form $L$ in $R$, it follows from Lemma~\ref{Reduction Number_2} that
$$
\H(R/(I,L),t)=0 \quad \text{for} \quad t\ge d.
$$
For such $t$ with the following exact sequence
$$
\begin{array}{lllllllllllllll}
0 & \ra & ((I:L)/I)_{t-1}(-1) & \ra & (R/I)_{t-1}(-1) & \overset{\times L}{\ra} & (R/I)_t & \ra & 0, 
\end{array}
$$
we have
$$
h_{t-1}=h_{t}+\dim_k ((I:L)/I)_{t-1}\ge h_{t}.
$$
So $h_{d-1}\ge h_d \ge h_{d+1}$, which is not the case.
Thus, we now assume that $r_1(A) \geq d$.

Suppose that $A$ is level and let $L$ be a general linear form in $R=k[x_1,x_2,x_3]$. Now consider the exact sequence
\begin{equation} \label{EQ:311-1}
\begin{array}{llllllllllllll}
0 & \ra & ((I: L)/I)_{d-1}(-1) & \ra & (R/
I)_{d-1}(-1) & \overset{\times L}{\ra} & (R/I)_{d} & \ra &
(R/(I, L))_{d} & \ra & 0.
\end{array}
\end{equation}
Since $d\leq r_1(A)$, we see that $\dim (R/(I, L))_{d}>0$, and so
$$
\begin{array}{llllllllllllllllll}
\dim ((I: L)/I)_{d-1}
& = & h_{d-1}-h_d +\dim (R/(I, L))_{d} & (\text{by equation~\eqref{EQ:311-1}})\\
&\geq & \dim (R/(I, L))_{d} >0 & (\text{since } h_{d-1}\geq h_d).
\end{array}
$$
Moreover, since $A$ is level, we have
$$
\begin{array}{llllllllllllll}
0
& < & \dim ((I: L )/I)_{d-1} \\
& \le & 2\dim ((I: L)/I)_{d} & (\text{by Lemma~\ref{L:006} (a)})\\
& = & 2\dim ((\Gin(I): x_3)/\Gin(I))_{d} & (\text{by Remark~\ref{R:20101107-212}})\\
& = & 2\beta_{2,d+3}({\rm Gin}(I)) & (\text{by Lemma~\ref{L:013}})\\
& \le & 2\beta_{2,d+3}(I^{\rm lex}) & (\text{by the theorem of BHP in \cite{Bi, Hul, Pa}}).
\end{array}
$$
Thus, by Lemma~\ref{L:20101104-209}, $A$ has a socle element in degree $d$, which is a contradiction.

\medskip
\ni
{\be Case 4.} If $h_{d-1}< h_{d}>h_{d+1}$, then 
$$
\beta_{2,d+3}(I^{\rm lex}) =h_{d}-h_{d+1}+((h_{d+1})_{(d+1)})^{-1}_{\pp0}>0.
$$
Hence, by Lemma~\ref{L:20101104-209}, $A$ is not level.

\medskip

\ni
{\be Case 5.} Suppose $h_{d-1}< h_{d}=h_{d+1}$. If $h_d\leq d$ then $((h_{d})_{(d)})^{-1}_{\pp0}=0$, and thus
$$
0 < \beta_{2,d+2}(I^{\rm lex})=h_{d-1}-h_{d}+((h_{d})_{(d)})^{-1}_{\pp0}<0,
$$
which is impossible. Hence $h_d\ge d+1$. 

If $h_d=d+1$, then $((h_{d})_{(d)})^{-1}_{\pp0}=1$, and so
$$
\begin{array}{llllllllllllllllll}
h_d
& > & h_{d-1} \\
& = & h_d-((h_{d})_{(d)})^{-1}_{\pp0}+\beta_{2,d+2}(I^{\rm lex}) & (\text{by Lemma~\ref{L:20101104-208} (a)})\\
& = & (d+1)-1+\beta_{2,d+2}(I^{\rm lex})\\
& \ge & d+1 & (\text{since } \beta_{2,d+2}(I^{\rm lex})>0)\\
& = & h_d,
\end{array}
$$
which is impossible. Thus we have $h_d\geq d+2$, that is, $((h_{d+1})_{(d+1)})^{-1}_{\pp0}\geq 1$. Then, by Lemma~\ref{L:20101104-208} (a), we obtain
$$
\beta_{2,d+3}(I^{\rm lex}) =h_{d}-h_{d+1}+((h_{d+1})_{(d+1)})^{-1}_{\pp0}>0.
$$
Therefore, by Lemma~\ref{L:20101104-209}, $A$ is not level, which completes the proof.
\end{proof}

The following example shows that there exists an Artinian level $O$-sequence which satisfies the condition $h_{d-1}= h_d = h_{d+1}=d+1$.

\begin{exmp} \label{EX:20101107-302} Let $I=(x_1^2, x_2^{3})+(x_1,x_2,x_3)^7$. Then the Hilbert function of $R/I$ is
$$
(1,3,5,6,6,6,6) 
$$
and the reduction number $r_1(R/I)$ is
$$
\min\{ \ell \mid x_2^{\ell+1}\in I\}=2.
$$

Moreover, it is immediate that
$$
{\soc}(R/I)=(R/I)_6,
$$
and so the minimal free resolution of $R/I$ is
$$
\begin{array}{lllllllllllllllll}
0 & \ra &  R(-9)^6 \ \  \ra \ \ R(-5) \oplus  R(-8)^{12} 
  \ \ \ra \ \  R(-2) \oplus  R(-3) \oplus  R(-7)^6 \\
  & \ra &  R \ \ \ra \ \ R/I \ \ra \ \ 0.  
\end{array}
$$

Note that the minimal free resolution of $R/I^{\rm lex}$ is 
$$
\begin{array}{lllllllllllllllllllllll}
0 & \ra &  R(-6)\oplus R(-{\mathbf 7})\oplus R^6(-9) \\
& \ra &  R(-4)\oplus R^2(-5)\oplus R^2(-6)\oplus R(-{\mathbf 7})\oplus R^{12}(-8) \\
& \ra &  R(-2)\oplus R(-3)\oplus R(-4)\oplus R(-5)\oplus R(-6)\oplus R^6(-7) & \ra &  R & \ra & R/I^{\rm lex} & \ra & 0 . 
\end{array} 
$$
This means that $R/I$ is an Artinian level algebra with the condition $h_4=h_5=h_6=5+1$, and 
$$
\beta_{1,5+2}(I^{\rm lex})=\beta_{2,5+2}(I^{\rm lex})=1.
$$
\end{exmp}

\begin{rem} In Example~\ref{EX:20101107-302}, we constructed an Artinian level algebra $R/I$ which satisfied the condition $h_{d-1}=h_d=h_{d+1}=d+1$ and $\beta_{1,d+2}(I^{\rm lex})=\beta_{2,d+2}(I^{\rm lex})>0$. The following are other examples of Artinian level $O$-sequences which satisfy the condition $h_{d-1}<h_d<h_{d+1}$ and $\beta_{1,d+2}(I^{\rm lex})=\beta_{2,d+2}(I^{\rm lex})>0$. 
\end{rem} 

\begin{exmp}[CoCoA] \label{EX:20101110-401}
We provide two examples of our results via calculations done by CoCoA.
\begin{itemize}
\item[(a)] 
Consider a differentiable $O$-sequence
 $\H=(1, 3, 6, 10, {\mathbf {13}}, {\mathbf {15}}, {\mathbf {17}}, 19, 20)$ and  an Artinian algebra $R/I$ with Hilbert function $\H$. Then the minimal free resolution of $R/I^{\rm lex}$ is
$$
\begin{array}{lllllllllllllllllll} 
0 & \ra &  R(-{\mathbf 7}) \oplus R(-10) \oplus R^{20}(-11) \\
  & \ra &  R(-5) \oplus R^2(-6) \oplus R(-{\mathbf 7}) \oplus R^2(-9) \oplus R^{42}(-10) \\
  & \ra &  R^2(-4) \oplus R(-5) \oplus R(-6) \oplus R(-8) \oplus R^{22}(-9) 
  \ \ \ra \  \ R/I^{\rm lex} \ \ \ra \ \ 0 ,
\end{array}
$$   
and hence
$$
\beta_{2,7}(I^{\rm lex})=\beta_{1,7}(I^{\rm lex})=1.
$$
Moreover, the sequence $\H$ is a level $O$-sequence since any differentiable $O$-sequence can be a truncation of an Artinian Gorenstein $O$-sequence. 

\item[(b)] Here is another differentiable $O$-sequence $\H=(1, 3, 6, 10, {\mathbf{12}}, {\mathbf {14}}, {\mathbf {16}}, {18}, 19, 20)$, which is also a level $O$-sequence by the same argument as in (a). Furthermore, the minimal free resolution of $R/I^{\rm lex}$ is
$$
\begin{array}{lllllllllllllllllll} 
0  & \ra &  R(-{\mathbf 6})\oplus R(-10)\oplus R(-11)\oplus R^{20}(-12) \\
   & \ra &  R^3(-5)\oplus R(-{\mathbf 6})\oplus R^2(-9)\oplus R^2(-10)\oplus R^{42}(-11) \\
   & \ra &  R^3(-4)\oplus R(-5)\oplus R(-8)\oplus R(-9)\oplus R^{22}(-10) \ \ \ra \ \  R \ \ \ra \ \ R/I^{\rm lex} \ \ \ra \ \ 0,
\end{array}
$$
and thus
$$
\beta_{2,6}(I^{\rm lex})=\beta_{1,6}(I^{\rm lex})=1.
$$
\end{itemize}
\end{exmp}

 
\begin{rem} In Example~\ref{EX:20101110-401}, both examples show that $\Delta h_{d}=\Delta h_{d+1}$. From this observation, we obtain the following result. 
\end{rem}

\begin{cor} \label{T:20101110-404}
Let $A=R/I$ be an Artinian ring of codimension 3. Suppose that
$$\beta_{1,d+2}(I^{\rm lex})=\beta_{2,d+2}(I^{\rm lex})>0 \text{ for some } d<s.$$
If $A$ is level and $h_{d-1}<h_d<h_{d+1}$, then $\Delta h_{d}= \Delta h_{d+1}$.
\end{cor}

\begin{proof}
Note that it suffices to prove that $\Delta h_{d}= \Delta h_{d+1}$ for $h_{d-1}<h_d<h_{d+1}$. 

Suppose that $A$ is level. Using Lemma~\ref{L:20101104-209}, we see that 
\begin{equation}\label{eq:308}
\beta_{2,d+3}(I^{\rm lex})=0.
\end{equation}
Furthermore, it is a simple consequence of Eliahou-Kervaire (Theorem~\ref{T:012}) that
$\beta_{2,d+2}(I^{\rm lex})>0$ implies $\beta_{0,d}(I^{\rm lex})>0$. Hence we have
\begin{equation}\label{eq:308-1}
h_d<\binom{2+d}{2}.
\end{equation} 
Since $h_{d}<h_{d+1}$, one can easily check that $d+1< h_{d+1}$ and thus
$
\displaystyle d+1<h_{d+1}<\binom{2+(d+1)}{2}.
$
Then the $(d+1)$-{\em binomial expansion} of $h_{d+1}$ is of the form 
\begin{equation}\label{EQ:201}
(h_{d+1})_{(d+1)}:=\binom {1+(d+1)}{d+1}+\cdots+\binom {1+(d-(c-2))}{d-(c-2)}+\binom{d-(c-1)}{d-(c-1)}+\cdots+\binom{\delta}{\delta}  
\end{equation}
where $\delta\ge 1$. It follows from Lemma~\ref{L:20101104-208} (a) and (\ref{eq:308}) that
\begin{equation}\label{EQ:20101110-402}
\Delta h_{d+1}
  =  ((h_{d+1})_{(d+1)})^{-1}_{\pp0}-\beta_{2,d+3}(I^{\rm lex})=((h_{d+1})_{(d+1)})^{-1}_{\pp0}.
\end{equation}
Now we consider the case $c<d$ only in equation (\ref{EQ:201}). Indeed, if $c-1\leq d\leq c,$ then 
we have 
\[
(h_{d+1})_{(d+1)} =  \begin{cases}
\ds \binom {2+d}{d+1}+\cdots+\binom {3}{2}+\binom{2}{1}, & 
\text{if } d=c-1, 
\\[2ex]
\ds \binom {2+d}{d+1}+\cdots+\binom {4}{3}+\binom{3}{2}+\binom{1}{1}, & 
\text{if } d=c.
\end{cases}
\]
Using Pascal's identity and equation (\ref{EQ:20101110-402}) for both cases, we have 
\[h_d=\binom{2+d}{d},\]
which contradicts equation (\ref{eq:308-1}). Hence, by equation (\ref{EQ:20101110-402}), 
$$
\begin{array}{llllllllllllllllllllllllll}
h_d
& = & h_{d+1}-((h_{d+1})_{(d+1)})^{-1}_{\pp0}\\[.5ex] 
& = & \ds \phantom{\bigg[}\binom {1+d}{d}+\cdots+\binom {1+(d-(c-1))}{d-(c-1)}+\binom{d-(c-1)}{d-(c-1)}+\cdots+\binom{\delta}{\delta}\phantom{\bigg]},\\[2ex] 
& = & \begin{cases}
\ds \binom {1+d}{d}+\cdots+\binom {1+(d-(c-1))}{d-(c-1)}+\binom{1+(d-c)}{d-c}, & 
\text{if } \delta=1, 
\\[1.5ex]
\ds \binom {1+d}{d}+\cdots+\binom {1+(d-(c-1))}{d-(c-1)}+\binom{d-c}{d-c}+\cdots+\binom{\delta-1}{\delta-1}, & 
\text{if } \delta>1, 
\end{cases}
\end{array} 
$$
i.e.,
$$
\begin{array}{lllllllllllllllllllll}
{((h_d)_{(d)})}^1_1
& = & 
\begin{cases} \ds \binom {1+(d+1)}{(d+1)}+\cdots+\binom {1+(d-(c-2))}{(d-(c-2))}+\binom{1+(d-(c-1))}{(d-(c-1))}, & \text{if } \delta=1,\\[1.5ex]
\ds \binom {1+(d+1)}{(d+1)}+\cdots+\binom {1+(d-(c-2))}{(d-(c-2))}+\binom{d-(c-1)}{d-(c-1)}+\cdots+\binom{\delta}{\delta}, & \text{if } \delta>1. 
\end{cases}
\end{array} 
$$
Thus
$$ 
\begin{array}{llllllllllllll}
{((h_d)}_{(d)})^{1}_{1}-h_{d+1}
& = & \begin{cases}
1, & \text{if } \delta=1,\\
0, & \text{if } \delta>1.
\end{cases}	
\end{array}
$$
Moreover, by Lemma~\ref{L:20101104-208} (b),
$$
\begin{array}{llllllllllllllllllll}
0 & < & \beta_{1,d+2}(I^{\rm lex}) \\
& = & {((h_d)}_{(d)})^{1}_{1}+h_d-2h_{d+1}+((h_{d+1})_{(d+1)})^{-1}_{\pp0}\\
& = & {((h_d)}_{(d)})^{1}_{1} -h_{d+1}-\Delta h_{d+1} +((h_{d+1})_{(d+1)})^{-1}_{\pp0}\\
& = & {((h_d)}_{(d)})^{1}_{1} -h_{d+1} \quad (\text{by equation  \eqref{EQ:20101110-402}}).
\end{array}
$$ 
This means that
\begin{equation}\label{eq:308-2}
\beta_{1,d+2}(I^{\rm lex})={((h_d)}_{(d)})^{1}_{1}-h_{d+1}=1 \quad \text{ and } \quad \delta=1.
\end{equation}
In other words,  
$$
(h_d)_{(d)}=\binom {1+d}{d}+\cdots+\binom {1+(d-(c-1))}{d-(c-1)}+\binom{1+(d-c)}{(d-c)}.
$$
Hence,  $((h_{d+1})_{(d+1)})^{-1}_{\pp0}=c$ and $((h_{d})_{(d)})^{-1}_{\pp0}=c+1$, and so we obtain 
$$
\begin{array}{llllllllllllllllll}
\Delta h_d 
& = & ((h_{d})_{(d)})^{-1}_{\pp0}-\beta_{2,d+2}(I^{\rm lex})& (\text{by Lemma~\ref{L:20101104-208} (a)} )\\
& = & c+1-\beta_{1,d+2}(I^{\rm lex}) & (\text{since } \beta_{1,d+2}(I^{\rm lex})=\beta_{2,d+2}(I^{\rm lex})>0 )\\
& = & c & (\text{by equation (\ref{eq:308-2})} )\\
& = & ((h_{d+1})_{(d+1)})^{-1}_{\pp0}\\
& = & \Delta h_{d+1}, & (\text{by equation \eqref{EQ:20101110-402}} )
\end{array}
$$
as we wished.
\end{proof}

\begin{exmp} Let
$R/I$ be an Artinian ring with Hilbert function $\H=(1, 3, 6, 10, {\mathbf{15}}, {\mathbf{16}}, {\mathbf{18}}, 20)$. Then the minimal free resolution of $R/I^{\rm lex}$ is 
$$
\begin{array}{lllllllllllllllllll}
0 & \ra &  {R}^{\mathbf 2}({\bf -7}) \oplus R(-8) \oplus R^{20}(-10) \ \ \ra \ \  R^6(-6) \oplus R^2({\bf -7}) \oplus R(-8) \oplus R^{42}(-9) \\
  & \ra &  R^5(-5) \oplus R(-6) \oplus R(-7) \oplus R^{22}(-8) \ \ \ra \ \ R
  \ \ \ra \ \ R/I^{\rm lex} \ \ \ra \ \ 0.
\end{array}
$$ 
Thus 
$$
\beta_{2,7}(I^{\rm lex})=\beta_{1,7}(I^{\rm lex})=2
\quad \text{and} \quad \Delta h_5=1\ne 2=\Delta h_6.
$$
By Theorem~\ref{T:20101110-404}, any Artinian ring $R/I$ with Hilbert function $\H$ cannot be level. 
\end{exmp}

\section{$O$-sequences with The Condition $h_{d-1}=h_d<h_{d+1}$}

In this section, we consider Artinian $O$-sequences with the condition $h_{d-1}=h_d<h_{d+1}$. To describe an Artinian $O$-sequence with this condition, we begin  with the following lemma.

\begin{lem} \label{L:20101110-406}
Let $c$ and $d$ be positive integers satisfying $d<c<\binom{d+2}{2}$. Then
$$
(c_{(d)})^{-1}_{\pp0}-(c_{(d)})^{1}_{1}+c=0.
$$
\end{lem}
\begin{proof} Without loss of generality, we assume that
\[c_{(d)}=\binom {1+d}{d}+\cdots+\binom {1+d-\alpha}{d-\alpha}+\binom{d-(\alpha+1)}{d-(\alpha+1)}+\cdots+\binom{\delta}{\delta}.\]
Then we have
$$
\begin{array}{lllllllllll}
((c)_{(d)})^{-1}_{\pp0}
& = & \alpha+1, \quad \text{and}\\
((c)_{(d)})^{1}_{1} -c
& = & \alpha+1,
\end{array}
$$
and thus 
$$
((c)_{(d)})^{-1}_{\pp0}-{((c)}_{(d)})^{1}_{1}+c=0,
$$
as we wished.
\end{proof}

The following result is an useful criterion to determine if $A$ is level.

\begin{pro} \label{P:20101110-406}
Let $A=R/I$ be an Artinian  ring of codimension $3$ with Hilbert function $\H=(h_0,h_1,\dots,h_s)$. Suppose that $h_{d-1}= h_d < h_{d+1}$ for some $d<s$. Then $A$ is not level if $$((h_{d+1})_{(d+1)})^{-1}_{\pp0} \leq 2(\Delta h_{d+1}).$$ 
\end{pro}
\begin{proof} Since $h_{d-1}=h_d$, we get $h_d<\binom{2+d}{2}$. 
If $h_d\le d$, by Macaulay's Theorem we have $h_{d+1}\leq d=h_d$. So we may assume that $d< h_d< \binom{2+d}{2}$. Hence $((h_{d})_{(d)})^{-1}_{\pp0}>0$. 

Since $((h_{d+1})_{(d+1)})^{-1}_{\pp0} \leq 2(\Delta h_{d+1})$, we obtain that 
$$
\begin{array}{llllllllllllllllll}
\beta_{2,d+2}(I^{\rm lex}) 
& = & h_{d-1}-h_d +((h_{d})_{(d)})^{-1}_{\pp0} \quad (\text{by Lemma~\ref{L:20101104-208} (a)}) \\
& = & ((h_{d})_{(d)})^{-1}_{\pp0}\\
& = & ((h_{d})_{(d)})^{-1}_{\pp0}+\beta_{1,d+2}(I^{\rm lex})-{((h_d)}_{(d)})^{1}_{1}-h_d+2 h_{d+1}-((h_{d+1})_{(d+1)})^{-1}_{\pp0}\\
&   & (\text{by Lemma~\ref{L:20101104-208} (b)}) \\
& = & (((h_{d})_{(d)})^{-1}_{\pp0}-{((h_d)}_{(d)})^{1}_{1}+h_d)+(2\Delta h_{d+1}-((h_{d+1})_{(d+1)})^{-1}_{\pp0})+\beta_{1,d+2}(I^{\rm lex})\\
& \ge & (((h_{d})_{(d)})^{-1}_{\pp0}-{((h_d)}_{(d)})^{1}_{1}+h_d)+\beta_{1,d+2}(I^{\rm lex})\quad (\text{by Lemma~\ref{L:20101110-406}})\\
& = & \beta_{1,d+2}(I^{\rm lex}).
\end{array}
$$
If $\beta_{2,d+2}(I^{\rm lex})>\beta_{1,d+2}(I^{\rm lex})$, then $A$ has a socle element in degree $d-1$, which means $A$ is not level. If
$$
\beta_{2,d+2}(I^{\rm lex}) =\beta_{1,d+2}(I^{\rm lex})=((h_{d})_{(d)})^{-1}_{\pp0}>0,
$$
by Theorem~\ref{thm:4} $A$ is not level, which completes the proof. 
\end{proof}

\begin{exmp} Consider an $O$-sequence $\H=(1,3,6,10,{\mathbf{15}},{\mathbf{15}}, {\mathbf{16}})$. Then
$$
2=((16)_{(6)})^{-1}_{\pp0} \le 2 \Delta h_6=2.
$$
Therefore, by Proposition~\ref{P:20101110-406}, any Artinian algebra with Hilbert function $\H$ cannot be level. 
\end{exmp}

Before we construct an Artinian level $O$-sequence with the condition $((h_{d+1})_{(d+1)})^{-1}_{\pp0} > 2(\Delta h_{d+1})$, we introduce the theorem of Iarrobino to obtain a new level $O$-sequence from the given level-$O$-sequence. Moreover, let us recall the main facts of the theory of {\em inverse system}, or {\em Macaulay duality}, which will be a fundamental tool to build an example. For a complete description, we refer  to \cite{G} and \cite{I-K}. 

Let $S=k[y_1,\dots,y_n]$ and consider $S$ as a graded $R=k[x_1,\dots,x_n]$-module where the action of $x_i$ on $S$ is partial differentiation with respect to $y_i$. Then there is a one to one correspondence between graded Artinian algebras $R/I$ and finitely generated graded $R$-submodules $M$ in $S$, where $I={\rm Ann}(M)$ is the annihilator of $M$ in $R$, and conversely $M=I^{-1}$ is the $R$-submodules of $S$ which is annihilated by $I$.

\begin{thm}[Theorem 4.8A, \cite{Ia:1}] \label{T:20101110-407} Let $\H'=(h_0,h_1,\dots,h_s)$ be the $h$-vector of a level algebra $A=R/{\rm Ann}(M)$. Then, if $F$ is a general form of degree $s$, the level algebra $B=R/{\rm Ann}(\langle M,F\rangle)$ has the $h$-vector $\H=(H_0,H_1,\dots,H_s)$ where
$$
H_i=\min\bigg\{h_i+\binom{r-1+s-i}{s-i}, \binom{r-1+i}{i}\bigg\}
$$
for $i=1,\dots,s.$
\end{thm}

The following theorem shows that there is an Artinian level algebra whose Hilbert function satisfies the condition
$$
h_{d-1}=h_d <h_{d+1} \quad \text{and} \quad ((h_{d+1})_{(d+1)})^{-1}_{\pp0} > 2(\Delta h_{d+1}).
$$

\begin{thm}\label{T:20101110-409}
 Let ${\bf H}=(1, 3, h_2, \ldots, h_{d-1}, h_d, h_{d+1})$ be an $O$-sequence satisfying 
 $$h_{d-1}= h_d < h_{d+1}.$$ 
  \begin{itemize}
   \item[(a)] If $h_d\le 3d+2$, then $\H$ is not level. 
   \item[(b)] If $h_d\geq 3d+3$, then there exists an Artinian level algebra with the Hilbert function $\H$ for some value of $h_{d+1}$. 
  \end{itemize}
\end{thm}

\begin{proof} 
(a) {\be Case 1.} Suppose that $h_d< 3d$. Since 
$$
h_d\leq (3d-1)_{(d)}=\binom {1+d}{d}+\binom {d}{d-1}+\binom{d-2}{d-2}+\cdots+\binom{1}{1}, 
$$
and $h_{d+1}\leq {((h_d)}_{(d)})^{1}_{1}$, we see that $((h_{d+1})_{(d+1)})^{-1}_{\pp0}\leq 2$. Hence, 
$$
((h_{d+1})_{(d+1)})^{-1}_{\pp0}\leq 2\leq 2\Delta h_{d+1}.
$$
Therefore, by Proposition~\ref{P:20101110-406}, $\H$ cannot be a level $O$-sequence.\\

\medskip

\ni
{\be Case 2.} Suppose that $3d\leq h_d\le  3d+2$. If $h_d=3d$, then $h_{d-1}=h_d=3d<\binom{2+d}{d}$. Hence $d\ge 3$ and 
$$
(h_d)_{(d)}=\binom{1+d}{d}+\binom{d}{d-1}+\binom{d-1}{d-2}.
$$
This implies that
$$
h_{d+1}\le ((h_d)_{(d)})^1_1=\binom{2+d}{1+d}+\binom{1+d}{d}+\binom{d}{d-1},
\quad \text{that is,} \quad ((h_{d+1})_{(d+1)})^{-1}_{\phantom{-}0}\le 3. 
$$
By the similar argument as above, we obtain
$$
((h_{d+1})_{(d+1)})^{-1}_{\phantom{-}0}\le 3
$$
for $h_d=3d+1$ or $3d+2$ as well.

If $h_{d+1}\geq h_d+2$, i.e., $\Delta h_{d+1}\geq 2$, then we see that
$$
((h_{d+1})_{(d+1)})^{-1}_{\pp0}\le 3\leq 4\leq 2\Delta h_{d+1},
$$
and thus, by Proposition~\ref{P:20101110-406}, $A$ is not level. 

We now assume that $h_{d+1}=h_d+1$. Then, it follows from Lemma~\ref{L:20101104-208} that 
\[
\begin{array}{|c|c|c|c|}
\hline
 h_d & \quad 3d\quad & 3d+1 & 3d+2 \\
 \hline
 \hline
 \beta_{1, d+2}(I^{\rm lex}) & 3 & 3 & 4 \\
 \hline 
 \beta_{2, d+2}(I^{\rm lex})& 3 & 3 & 3 \\
 \hline
 \beta_{2, d+3}(I^{\rm lex})& 1 & 1 &  2\\
 \hline
\end{array}
\]
By Lemma~\ref{L:20101104-209} it is enough to prove that ${\bf H}$ is not level for the case where 
\[h_d=3d+2 \quad \text{ and }\quad h_{d+1}=3d+3.\] 
Assume that there is an  Artinian level algebra $A=R/I$ with Hilbert function ${\bf H}$. By Lemma~\ref{L:20101104-208}, the Betti diagram of $R/I^{\rm lex}$ is as follows.

\[
\begin{array}{c|lllllllllllllllll}
                              {\rm total} &1 & - & - & - &\\[1ex]
                             \hline
                                  0 &1 & . & . & . &\\[1ex]
                                   \cdots  &   & & \cdots& &\\[1ex]  
                                d-1& .& *& *&   3&\\[1ex]
                                  d  & . & 2 & 4 & 2&\\[1ex]
                                d+1& . &  *& * & * &\\[1ex]
\end{array}
\]

Let $J:=(I_{\le d+1})$. Note that $I^{\rm lex}$ and $J^{\rm lex}$ agree in degree $\le d+1$.  We then rewrite the Betti diagram of $R/J^{\rm lex}$ as follows.

$$
\begin{array}{c|lllllllllllllllll}
                              {\rm total} &1 & - & - & - &\\[1ex]
                             \hline
                                  0 &1 & . & . & . &\\[1ex]
                                   \cdots  &   & & \cdots& &\\[1ex]  
                                d-1& .& *& *&   3&\\[1ex]
                                  d  & . & 2 & 4 & 2&\\[1ex]
                                d+1& . &  a& b & * &\\[1ex]
\end{array}
$$
Since $R/I$ is level and $(I_{\le d+1})$ has no generators in degree $d+2$, we have
$$
0\leq a\leq 1 \quad \text{(by the cancellation principle)}.
$$
 
\medskip

\noindent {\em Case 2-1.} If $a=0$,  then by the result of Eliahou-Kervaire (Theorem~\ref{T:012}), we 
have $b=0$, which means  $R/(I_{\le d+1})$ has a two dimensional socle element
in degree $d$, so does $R/I$. This is a contradiction.

\medskip

\noindent
{\em Case 2-2.} If $a=1$, then $J^{\rm lex}$ has one generator in degree $d+2$.
Define 
$$
h_{d+2}:=\H(R/J^{\rm lex}, d+2).
$$
Then we have 
$$
\begin{array}{ccccccccccccccc}
h_{d+2}=((h_{d+1})_{(d+1)})^1_1 -1=((3d+3)_{(d+1)})^1_1 -1=3d+5, \text{ i.e.,}\\
((h_{d+2})_{(d+2)})^{-1}_{\pp0}=((3d+5)_{(d+2)})^{-1}_{\pp0}=2.
\end{array} 
$$
Hence, from Lemmas~\ref{L:013} and \ref{L:20101104-208} we have 
\begin{equation}\label{EQ:411}
\begin{array}{llllllllllllllllllll}
\dim_k ((J^{\rm lex}:x_3)/J^{\rm lex})_{d+1}
& = & \big|\left\{\, T \in \GG(J^{\rm lex})_{d+2} \,\big|\, x_3 \text{ divides }\, T\, \right\}\big|
\\[1ex]
& = & \beta_{2,d+4}(J^{\rm lex}) \\
& = & h_{d+1}-h_{d+2}+((h_{d+2})_{(d+2)})^{-1}_{\pp0} \\
& = & 0.
\end{array}
\end{equation}
Since $x_1^{d+2}\notin   \GG(J^{\rm lex})_{d+2}$, by  Theorem~\ref{T:012} and equation~\eqref{EQ:411}, we find
$$
b=\beta_{1,d+3}(J^{\rm lex})=\sum_{T\in \GG(J^{\rm lex})_{d+2}}\binom{m(T)-1}{1}=1.
$$
Using the cancellation principle, we know $R/J$  has at least one socle element in degree $d$. Since $R/I$ and $R/J$ agree in degree $\le d+1$, $R/I$ has also a socle element in degree $d$. This is a contradiction.

\medskip

(b) Applying  Theorem~\ref{T:20101110-407} to a differentiable $O$-sequence 
$$
{\bf H'}=(1,3,6,\dots,3(d-1)+(\ell -3),\overset{d\text{-th}}{3d+(\ell-3) }, 3(d+1)+(\ell -3))
$$ 
with $\ell \ge 3$, we obtain an Artinian level $O$-sequence
$$
\begin{array}{llllllllllllllllllllll}
H_{d-1}
& = & \ds \min\bigg\{3(d-1)+(\ell-3)+\binom{4}{2}, \binom{d+1}{2}\bigg\} & = & 3d+\ell, \\[1.5ex] 
H_{d}
& = & \ds \min\bigg\{3d+(\ell-3)+\binom{3}{1}, \binom{d+2}{2}\bigg\} & = & 3d+\ell, 
\quad \text{and} \\[1.5ex]
H_{d+1} 
& = & \ds \min\bigg\{3(d+1)+(\ell -3)+\binom{2}{0}, \binom{d+3}{2}\bigg\} & = & 3d+(\ell +1), 
\end{array}
$$
as we wished. 
\end{proof}

\begin{exmp} \label{EX:20101110-408} Consider a differentiable $O$-sequence $\H'=(1, 3, 6, 10, 15, 21, 28, 36, 45, 55, 58, 61, 64)$, which is an Artinian level $O$-sequence. By Theorem~\ref{T:20101110-407}, we can construct a new level $O$-sequence as follows.
$$
\H= (1, 3, 6, 10, 15, 21, 28, 36, 45, 55, {\mathbf{64}}, {\mathbf{64}}, {\mathbf{65}}),
$$ 
which satisfies the following two conditions
$$
h_{10}=h_{11}<h_{12} \quad \text{and} \quad 6={((65)_{(12)})}^{-1}_{\pp0}>2\Delta h_{12}=2.
$$
\end{exmp}
The above example~\ref{EX:20101110-408} also shows that there is an Artinian level algebra whose Hilbert function satisfies the conditions 
$$
h_{10}=h_{11}<h_{12} \quad \text{and} \quad 64=h_{11}>3d=3\cdot {11}=33.
$$

If we couple our previous work done in \cite{AS} with the results of the previous and this sections, we obtain the following result.

\begin{thm}  Let $R/I$ be an Artinian ring of codimension $3$ with Hilbert function $\H=(h_0,h_1,\dots,h_{d+1})$. Then,
\begin{itemize}
\item[(a)] if $h_{d-1}>h_d=h_{d+1}$ with $h_{d}\le 2d+3$, then $R/I$ is not level,
\item[(b)] if $h_{d-1}>h_d=h_{d+1}$ with $h_{d}\ge 2d+4$, the $R/I$ is  level for some value of $h_{d-1}$,
\item[(c)] if   $h_{d-1}=h_d<h_{d+1}$ with $h_d\le 3d+2$, then $R/I$ is not level,
\item[(d)] if   $h_{d-1}=h_d<h_{d+1}$ with $h_d\ge 3d+3$, then $R/I$ is level for some value of $h_{d+1}$,
\item[(e)] if $R/I$ is level and $\beta_{1,d+2}(I^{\rm lex})=\beta_{2,d+2}(I^{\rm lex})$, then 
\begin{itemize}
 \item[(i)] $h_{d-1}= h_d = h_{d+1}=d+1$, or
 \item[(ii)] $h_{d-1}<h_d<h_{d+1}$ and $\Delta h_{d}= \Delta h_{d+1}$."
 \end{itemize}
 \end{itemize}
\end{thm}

\section*{Acknowledgement}

The authors are truly thankful to the reviewer, whose comments and suggestions enabled us to make substantial improvements in the paper. 

\bibliographystyle{amsalpha}

\end{document}